\documentclass[
final
 , nomarks
]{dmtcs-episciences}

\usepackage[utf8]{inputenc}
\usepackage{subfigure}
\usepackage{amsmath, amsthm, amssymb, xspace}

\newtheorem{thm}{Theorem}
\newtheorem{lem}[thm]{Lemma}
\newtheorem{cor}[thm]{Corollary}
\theoremstyle{definition}
\newtheorem{con}{Conjecture}
\theoremstyle{remark}
\newtheorem{case}{Case}

\newcommand*{\lemma}[1]{Lemma~\ref{#1}}
\newcommand*{\theorem}[1]{Theorem~\ref{#1}}
\newcommand*{\conj}[1]{Conjecture~\ref{#1}}
\newcommand*{\sect}[1]{Section~\ref{sec:#1}}

\newcommand*{\prob}[1]{\mathbb{P}(#1)}
\newcommand*{\probb}[1]{\mathbb{P}\bigl(#1\bigr)}
\newcommand*{\probx}[1]{\mathbb{P}\Bigl(#1\Bigr)}
\newcommand*{\cmp}[1]{#1^{\scriptscriptstyle\complement}}
\newcommand*{\modn}[1][n]{mod\;$#1$\xspace}
\newcommand*{\abs}[1]{\lvert#1\rvert}
\newcommand*{\floor}[1]{\lfloor#1\rfloor}
\newcommand{\ie}{i.e.\ }
\newcommand{\eg}{e.g.\ }

\usepackage[round]{natbib}

\author{John Haslegrave}
\title{Proof of a local antimagic conjecture}
\affiliation{
  Mathematics Institute\\
  University of Warwick\\
  Coventry, UK}
\received{2017-8-29}
\revised{2018-4-30}
\accepted{2018-5-14}
\begin{document}
\publicationdetails{20}{2018}{1}{18}{3887}

\maketitle
\begin{abstract}
An \textit{antimagic labelling} of a graph $G$ is a bijection $f:E(G)\to\{1,\ldots,\lvert E(G)\rvert\}$ such that the sums $S_v=\sum_{e\ni v}f(e)$ distinguish all vertices $v$. A well-known conjecture of Hartsfield and Ringel (1994) is that every connected graph other than $K_2$ admits an antimagic labelling. Recently, two sets of authors (Arumugam, Premalatha, Ba\v ca \& Semani\v cov\'a-Fe\v nov\v c\'ikov\'a (2017) and  Bensmail, Senhaji \& Szabo Lyngsie (2017)) independently introduced the weaker notion of a \textit{local antimagic labelling}, where only adjacent vertices must be distinguished. Both sets of authors conjectured that any connected graph other than $K_2$ admits a local antimagic labelling. We prove this latter conjecture using the probabilistic method. Thus the parameter of local antimagic chromatic number, introduced by Arumugam et al., is well-defined for every connected graph other than $K_2$.
\end{abstract}

\section{Introduction}
Fix a graph $G$ with $m$ edges and assign labels from $\{1,\ldots,m\}$ to the edges of $G$, with each label appearing exactly once. For a vertex $v$, let $S_v$ be the sum of the labels of edges meeting $v$. If $S_v\neq S_w$ for every pair of distinct vertices $v,w$, then the labelling is an \textit{antimagic labelling} of $G$. If such a labelling exists, $G$ is said to be \textit{antimagic}. Clearly $K_2$ is not antimagic, since the two vertices have the same set of incident edges (and the same argument applies to any graph which has $K_2$ as a component). Hartsfield and Ringel introduced antimagic labellings in \cite{HR94}, and conjectured that every connected graph other than $K_2$ is antimagic. Since then the conjecture has received much attention, and has been proved for several special classes of graphs. Perhaps the most striking result to date is the proof by Cranston, Liang and Zhu (\cite{CLZ}) for regular graphs of odd degree, which was subsequently adapted to regular graphs of even degree by B\'erczi, Bern\'ath and Vizer (\cite{BBV}) and by Chang, Liang, Pan and Zhu (\cite{CLPZ}). Notably, however, the conjecture is still unsolved even for some particularly simple natural graph classes such as trees. For further details of known classes of antimagic graphs, see Section~6 of \cite{Gal97}.

Recently, two groups of authors independently considered the weaker condition of neighbour-distin\-guishing, rather than vertex-distinguishing, labellings. A labelling using every label from $\{1,\ldots,m\}$ exactly once is a \textit{local antimagic labelling} if the sums $S_v$ and $S_w$ defined above are distinct for every pair of neighbouring vertices $v,w$.

Arumugam, Premalatha, Ba\v ca and Semani\v cov\'a-Fe\v nov\v c\'ikov\'a were motivated by using the sums $S_v$ to define a proper colouring of $G$, and in \cite{APBS} introduced the \textit{local antimagic chromatic number}, $\chi_{la}(G)$, of a graph $G$ as the minimum number of colours among all proper colourings of $G$ which arise in this way. Bensmail, Senhaji and Szabo Lyngsie (\cite{BSSL}) were motivated by the 1-2-3 conjecture, posed in \cite{KLT} by Karo\'nski, \L uczak and Thomason. This states that if $G$ is any graph with no isolated edge (that is, no component isomorphic to $K_2$), then the edges of $G$ may be labelled, using labels from $\{1,2,3\}$ without restriction on the number or proximity of occurrences of a given label, in such a way that the vertex sums distinguish neighbouring vertices. For a survey of results in this direction, see \cite{Sea12}. Of course, for any constant number of permitted labels, distinguishing vertices in some local sense is the best we can hope for, since if $G$ has sufficiently many low-degree vertices then any labelling will fail to distinguish some two of them by the pigeonhole principle. For example, if the edges of $C_6$ are labelled from $\{1,2,3\}$, there are only $5$ possible vertex sums and so not all vertices will be distinguished; however, labelling the edges $1,2,3,1,2,3$ distinguishes all neighbours. Increasing the number of permitted labels and requiring all vertices to be distinguished gives the concept of \textit{irregularity strength}, introduced in  \cite{irregular} by Chartrand, Jacobson, Lehel, Oellermann, Ruiz and Saba; the irregularity strength of a graph $G$ is the smallest positive integer $k$ for which there is a labelling of the edges of $G$ with labels from $\{1,\ldots,k\}$ where every vertex has a different sum.

Neighbour-distinguishing labellings have also been considered in the context of proper edge-labellings, \ie any two edges which share a vertex must be given different labels. Zhang, Liu and Wang (\cite{ZLW}) conjectured that this can be done using only $\Delta(G)+2$ different (not necessarily consecutive) labels for any connected graph $G$ with at least $6$ vertices. A bound of $\Delta(G)+O(1)$ has been given in \cite{Hat05}. This local restriction on labels may be thought of as intermediate between 1-2-3-type problems, where there is no restriction on how labels may be assigned, and antimagic-type problems, where there is a global restriction that no two edges may share a label.

\cite{APBS} and \cite{BSSL} made almost identical conjectures on local antimagic labellings, although Arumugam et al.\ considered only connected graphs. The slightly stronger form of the conjecture, given by Bensmail et al., is that every graph without isolated edges has a local antimagic labelling. Bensmail et al.\ proved that the conjecture is true for trees. 

We prove the stronger form of the conjecture. Since $\chi_{la}(G)$ is well-defined precisely when $G$ admits a local antimagic labelling, our main result shows that $\chi_{la}(G)$ exists for every graph $G$ with no isolated edge. 

Our approach is to show that assigning the numbers $1,\ldots,m$ to the edges of $G$ according to a random permutation distinguishes all neighbours with positive probability. Although our proof is not constructive, and so does not give a deterministic algorithm for finding a local antimagic labelling, we give bounds on this probability in terms of $m$, which allow us to conclude that the simple algorithm of generating a random permutation, checking whether it yields a local antimagic labelling, and repeating until one is found, has expected runtime at most quadratic in $m$.

Our methods actually prove the stronger result that for every pair of neighbours the probability that a random permutation fails to distinguish them is at most $1/m$, and that equality is only possible in one exceptional small case. The result follows immediately by the union bound. 

\section{Parity of random subsets}
We will need the following preparatory lemma bounding the discrepancy between the probability that the sum of a random subset of fixed size is odd and the probability that it is even. Here, and throughout the paper, we use the notation $[n]:=\{1,\ldots,n\}$.
\begin{lem}\label{parity}Suppose $0<c<n$. Let $X$ be chosen uniformly at random from all $c$-element subsets of $[n]$. Then, for each $i\in\{0,1\}$,
\[\probx{\sum_{x\in X}x\equiv i\mod2}\leq\begin{cases}\frac12\quad&n\text{ even\,, }c\text{ odd}\,;\\
\frac12\Bigl(1+\frac{1}{n-1}\Bigr)\quad&n\text{ even\,, }c\text{ even}\,;\\
\frac12\Bigl(1+\frac{1}{n}\Bigr)\quad&n\text{ odd}\,.\end{cases}\]\end{lem}
\begin{proof}Divide the elements of $[n]$ into $\floor{n/2}$ pairs of the form $(2i-1,2i)$. For each $c$-element subset $Y$, define $Y^*$ to be the subset given by finding the first pair $P$ such that $\abs{Y\cap P}=1$ and replacing the element of $P$ which is in $Y$ with the other one. If no such pair exists, set $Y^*=Y$. Clearly $(Y^*)^*=Y$, and if $Y^*\neq Y$ then $Y$ and $Y^*$ have opposite-parity sums. So of the subsets $Y$ with $Y^*\neq Y$, exactly half have the desired parity. Write $f(n,c)$ for the number of subsets $Y$ with $Y^*=Y$. Then we have
\begin{align*}\probx{\sum_{x\in X}x\equiv i\mod2}&\leq\frac{\frac12\Bigl(\binom{n}{c}-f(n,c)\Bigr)+f(n,c)}{\binom{n}{c}}\\
&=\frac12\biggl(1+f(n,c)\Bigm/\binom{n}{c}\biggr)\,.\end{align*}

Any $Y$ with $Y^*=Y$ must consist of $\floor{c/2}$ pairs, together with the unpaired element if $c$ is odd. Since there is no unpaired element if $n$ is even, whenever $n$ is even and $c$ odd we have $f(n,c)=0$.

If $n$ and $c$ are both even, we have $f(n,c)=\binom {n/2}{c/2}$. It suffices to show that 
\begin{equation}\binom{n}{c}\Bigm/\binom{n/2}{c/2}\geq n-1\,,\label{eveneven}\end{equation}
and by symmetry of binomial coefficients we may assume $c\leq n/2$. \eqref{eveneven} is an equality when $c=2$, and if $2<c\leq n/2$ then
\[\frac{\binom{n}{c}/\binom{n/2}{c/2}}{\binom{n}{c-2}/\binom{n/2}{c/2-1}}=\frac{n-c+1}{c-1}>1\,,\]
and so, by induction, the result follows for all $c$ in this range.

If $n$ is odd then, by considering the complement $\cmp X$ (and swapping the required parity if $\binom{n+1}2$ is odd), it suffices to consider the case where $c$ is also odd. Now we have $f(n,c)=\binom{(n-1)/2}{(c-1)/2}$, and so need to show that $\binom{n}{c}/\binom{(n-1)/2}{(c-1)/2}\geq n$. This is an equality when $c=1$. If $c>1$ then $\binom{n}{c}=\frac{n}{c}\binom{m-1}{c-1}\geq\frac n{n-2}\binom{n-1}{c-1}$ and by \eqref{eveneven} we have $\binom{n-1}{c-1}/\binom{(n-1)/2}{(c-1)/2}\geq n-2$, giving the result.
\end{proof}

\section{Bounds for individual edges}
Fix a graph $G$ (with no isolated edge), and write $m=\abs{E(G)}$. We use a random permutation of $[m]$ to label the edges of $G$, and then for each vertex $x$ define $S_x$ to be the sum of labels of edges meeting $x$.

The aim of this section is to show that for each edge $xy$ we have $\prob{S_x=S_y}\leq\frac 1m$, and that the inequality is strict except when $m=5$ and $d(x)=d(y)=3$. We deal with three types of edge separately: edges where the vertex degrees differ by $1$ or $2$; edges where the vertex degrees are equal; and edges where the vertex degrees differ by more than $2$. A different method is needed in each case. Although the case of a difference of $1$ or $2$ would naturally be the middle case, we consider that case first as it is the simplest and some results from it are used in later cases.

Rather than thinking directly about labellings of $G$, we consider the following situation. Let $\sigma$ be a random permutation of $[n]$, and let $a,b$ be fixed positive integers with $a+b<n$. Define \[D_{a,b}(\sigma):=\sum_{i=1}^a\sigma(i)-\sum_{i=a+1}^{a+b}\sigma(i)\,.\] 
Let $t$ be a fixed target. (Often, but not always, we have $t=0$.) Then we will bound the probability
\[p_{n;a,b}(t):=\prob{D_{a,b}(\sigma)=t}\,.\]

The purpose of this definition is that $S_x-S_y$ has the same distribution as $D_{a,b}(\sigma)$, where $n=m$, $a=d(x)-1$ and $b=d(y)-1$. The first $a$ values of $\sigma$ correspond to the labels of the edges meeting $x$ but not $y$, and the next $b$ correspond to the labels of edges meeting $y$ but not $x$. The label on the edge $xy$ does not affect $S_x-S_y$, since it contributes the same amount to each sum. 

Throughout this section we therefore assume that $a$ and $b$ are not both zero, since that would correspond to the case $d(x)=d(y)=1$, which is not permitted as $G$ has no isolated edges, and further that they are both positive (since if $x$ and $y$ are adjacent and exactly one has degree $1$ then no labelling can fail to distinguish them). Also we assume that $a+b<n$, since $d(x)+d(y)\leq m+1$.

\subsection{Degrees differing by $1$ or $2$}
In this section we prove results about $D_{a,b}(\sigma)$ \modn, which will be useful later. First we consider values differing by $1$.
\begin{lem}\label{diffone}For any $t$, $\prob{D_{a,a-1}(\sigma)\equiv t\mod n}=\frac1n$.\end{lem}
\begin{proof}First choose a random permutation $\tau$, and a random number $r\in[n]$, and then let $\sigma$ be the permutation with $\sigma(i)\equiv\tau(i)+r$ \modn for every $i$. This process still produces a uniformly random $\sigma$. For fixed $\tau$, $D_{a,a-1}(\sigma)\equiv D_{a,a-1}(\tau)+r$ \modn, and so $\prob{D_{a,a-1}(\sigma)\equiv t\mid \tau}=\frac1{n}$; the result follows.\end{proof}

We prove a similar result for values differing by $2$.
\begin{lem}\label{difftwo}For any $t$, $\prob{D_{a,a-2}(\sigma)\equiv t\mod n}\leq\frac1{n-1}$.\end{lem}
\begin{proof}As in the proof of \lemma{diffone}, we first choose a random permutation $\tau$ and a random $r\in[n]$, and define $\sigma(i)\equiv\tau(i)+r$ \modn. Now $D_{a,a-2}(\sigma)\equiv D_{a,a-2}(\tau)+2r$ \modn. If $n$ is odd, each choice of $r$ gives a different residue \modn, and so $\prob{D_{a,a-2}(\sigma)\equiv t\mod n}=\frac1n$. If $n$ is even then 
\[\prob{D_{a,a-2}(\sigma)\equiv t\mod n\mid\tau}=\begin{cases}\frac2n\quad&\text{if }\sum_{i=1}^{2a-2}\tau(i)\equiv t\mod 2\,;\\
0\quad&\text{otherwise.}\end{cases}\]
By \lemma{parity}, the first case applies with probability at most $\frac12\bigl(1+\frac1{n-1}\bigr)$ and so $\prob{D_{a,a-2}(\sigma)\equiv t\mod n}\leq\frac1n\cdot\frac n{n-1}$, as required.\end{proof}

We are now ready to prove the required bounds for edges of this type.
\begin{thm}\label{onetwo}$\prob{D_{a,a-1}(\sigma)=0}<\frac1n$ and $\prob{D_{a,a-2}(\sigma)=0}\leq\frac1{2(n-1)}$\end{thm}
\begin{proof}For the first statement we know from \lemma{diffone} that $\prob{D_{a,a-1}(\sigma)\equiv 0\mod n}=\frac1n$, and so it suffices to show that $\prob{D_{a,a-1}(\sigma)=n}>0$. If $2a-2$ is a multiple of $4$, we may construct a $\sigma$ with $D_{a,a-1}(\sigma)=n$ by setting $\sigma(1),\ldots,\sigma(a-1)$ to be the numbers in $[2a-2]$ congruent to $2$ or $3$ \modn[4], setting $\sigma(a+1),\ldots,\sigma(2a-1)$ to be the remaining numbers in $[2a-2]$, and setting $\sigma(a)=n$. If $2a-2$ is not a multiple of $4$, the construction is the same except $\sigma(a)=n-1$ (since $n>a+(a-1)$, $n-1\not\in[2a-2]$).

For the second statement, we know from \lemma{difftwo} that $\prob{D_{a,a-2}(\sigma)\equiv 0\mod n}\leq\frac1{n-1}$. If $\sigma$ is a permutation with $D_{a,a-2}(\sigma)=0$, consider the permutation $\sigma^*$ defined by $\sigma^*(i)=n-\sigma(i)$ if $\sigma(i)\neq n$, and $\sigma^*(i)=n$ if $\sigma(i)=n$. The mapping $\sigma\mapsto\sigma^*$ is injective, and 
\[D_{a,a-2}(\sigma^*)=\begin{cases}2n&\quad n\not\in[2a-2]\,;\\3n&\quad n\in[a]\,;\\n&\quad n\in\{a+1,\ldots,2a-2\}\,,\end{cases}\]
so $D_{a,a-2}(\sigma^*)\equiv 0$ \modn, but $D_{a,a-2}(\sigma^*)\neq 0$. Thus $\prob{D_{a,a-2}(\sigma)=0}\leq\frac12\prob{D_{a,a-2}(\sigma)\equiv 0\mod n}$, as required.\end{proof}

\subsection{Equal degrees}\label{sec:equal}
We first consider the case of equal-degree vertices which meet all edges.
\begin{lem}\label{basecase}We have $p_{2a+1;a,a}(t)\leq\frac 1{2a+1}$ for every $t$.\end{lem}
\begin{proof}We may choose a uniformly random permutation $\sigma$ by first choosing a uniformly random permutation $\tau$ and a uniformly random $r\in[a+1]$, then setting $\sigma(r)=\tau(a+1)$, $\sigma(a+1)=\tau(r)$, and $\sigma(i)=\tau(i)$ for $i\neq r,a+1$. (Note that in one case this has $\sigma\equiv\tau$.)

For each choice of $\tau$, we have $D_{a,a}(\sigma)=t$ for at most one of the $a+1$ possible values of $r$. Furthermore, if $\tau$ is such that $\tau(1)+\cdots+\tau(2a)\not\equiv t$ \modn[2], no choice of $r$ works, since in each case we have $D_{a,a}(\sigma)\equiv\tau(1)+\cdots+\tau(2a)$ \modn[2]. Since $\tau(1)+\cdots+\tau(2a)\equiv t$ if and only if $\tau(n)\equiv\binom {2a+2}2+t$ (\modn[2]), we have $\prob{\tau(1)+\cdots+\tau(2a)\equiv t\mod 2}=\frac{a}{2a+1}$ or $\frac{a+1}{2a+1}$, depending on the parity of $t$. So certainly $p_{2a+1;a,a}(t)\leq\frac{a+1}{2a+1}\cdot\frac{1}{a+1}=\frac 1{2a+1}$.
\end{proof}

Next we use an inductive argument to prove the required bound for edges between vertices of equal degree. 
\begin{thm}\label{equal}For all $n>2a$ we have $p_{n;a,a}(0)\leq\frac1n$. Furthermore, unless $a=2$ and $n=5$ the inequality is strict, and unless $n=2a+1$ we have $\frac1n-p_{n;a,a}(0)\geq\frac{1}{3n(n-1)}$.\end{thm}
\begin{proof}We show this by induction on $n$ (for fixed $a$). We may assume $a>1$ since otherwise trivially $p_{n;a,a}(0)=0$. 

The base case is implied by \lemma{basecase}, except that we still need to check that the inequality is strict for $a>2$. In the case $a=3$ we have $\binom 82$ even, and so in the proof of \lemma{basecase} we get $\prob{\tau(1)+\cdots+\tau(2a)\equiv 0\mod 2}=\frac{a}{2a+1}$, giving strict inequality. If $a\geq 4$ then by choosing $\tau$ to maximise $D_{a,a}(\tau)$ (while still ensuring that $\tau(1)+\cdots+\tau(2a)$ is even) we can ensure $D_{a,a}(\tau)>4a$, which implies that for this $\tau$, no value of $r$ gives $D_{a,a}(\sigma)=0$. So again we have a strict inequality. 

For $n>2a+1$, choose a random permutation $\tau$. With probability $\frac{n-2a}{n}$ we have $n\not\in\{\tau(1),\ldots,\tau(2a)\}$. Conditional on this event, the probability that 
$\sum_{i=1}^a\tau(i)=\sum_{i=a}^{2a}\tau(i)$ is simply $p_{n-1;a,a}(0)$, which is at most $\frac1{n-1}$ by the induction hypothesis. Conversely, we have $n\in\{\tau(1),\ldots,\tau(2a)\}$ with probability $\frac{2a}{n}$, and conditional on this event the probability is $p_{n-1;a,a-1}(n)$ (without loss of generality, we may assume that $n\in\{\tau(a+1),\ldots,\tau(2a)\}$). Therefore
\begin{equation}
p_{n;a,a}(0)\leq\frac{n-2a}{n}\cdot\frac{1}{n-1}+\frac{2a}{n}p_{n-1;a,a-1}(n)\,.\label{inout}
\end{equation}
Now we bound $p_{n-1;a,a-1}(n)$. Letting $\rho$ be a uniformly random permutation of $[n-1]$, and writing $A$ for the event $1\in\{\rho(a+1),\ldots,\rho(2a-1)\}$, we have
\begin{equation}p_{n-1;a,a-1}(n)=\prob{(D_{a,a-1}(\rho)=n)\wedge A}+\prob{(D_{a,a-1}(\rho)=n)\wedge\cmp A}\,.\label{AnotA}\end{equation}
Conditioning on $A$, we may remove $1$ to get a uniformly random bijection $\rho':[n-2]\to\{2,\ldots,n-1\}$. We see that $D_{a,a-1}(\rho)=n$ if and only if $D_{a,a-2}(\rho')=n+1$. Defining $\rho''(i)=\rho'(i)-1$ for each $i$, $\rho''$ is a random permutation of $[n-2]$ and $D_{a,a-2}(\rho'')=D_{a,a-2}(\rho')-2$. Consequently
\[\prob{D_{a,a-1}(\rho)=n\mid A}=p_{n-2;a,a-2}(n-1)\,,\]
and this is at most $\frac1{n-3}$ by \lemma{difftwo}. Thus
\begin{equation}\prob{(D_{a,a-1}(\rho)=n)\wedge A}\leq\frac1{n-3}\prob{A}=\frac{1}{n-3}\cdot\frac{a-1}{n-1}<\frac1{2(n-1)}\,,\label{Aholds}\end{equation}
since $2a-2<n-3$.

Now we may bound $\prob{(D_{a,a-1}(\rho)=n)\wedge\cmp A}$ in two stages: first we bound $\prob{(D_{a,a-1}(\rho)\equiv n\mod{(n-1)})\wedge\cmp A}$, and then $\prob{D_{a,a-1}(\rho)= n\mid (D_{a,a-1}(\rho)\equiv n\mod{(n-1)})\wedge\cmp A}$. Note that, using \lemma{diffone},
\begin{align}
\prob{(D_{a,a-1}(\rho)\equiv n\mod{(n-1)})\wedge\cmp A}&=\prob{D_{a,a-1}(\rho)\equiv n\mod{(n-1)}}\nonumber\\
&\qquad{}-\prob{(D_{a,a-1}(\rho)\equiv n\mod{(n-1)})\wedge A}\nonumber\\
&=\frac{1}{n-1}-\prob{(D_{a,a-1}(\rho)\equiv n\mod{(n-1)})\wedge A}\nonumber\\
&\leq\frac{1}{n-1}-\prob{(D_{a,a-1}(\rho)=n)\wedge A}\,.\label{comp}
\end{align}
Now suppose $\rho$ is a permutation for which $D_{a,a-1}(\rho)=n$ and $\cmp A$ holds. Define the permutation $\rho^*$ by 
\[\rho^*(i)=\begin{cases}n+1-i\quad&\text{if }\rho(i)\neq 1\,;\\1\quad&\text{if }\rho(i)=1\,.\end{cases}\]
Now for every such $\rho$ we have a different $\rho^*$, and clearly $\cmp A$ also holds for $\rho^*$, since the position of $1$ does not change. We claim that $\rho^*$ satisfies
$D_{a,a-1}(\rho^*)\equiv n$ \modn[(n-1)], but $D_{a,a-1}(\rho^*)\neq n$. This is because we have
\[D_{a,a-1}(\rho^*)=\begin{cases}n+1-D_{a,a-1}(\rho)\quad&\text{if }1\in\{\rho(2a),\ldots,\rho(n)\}\\
2-D_{a,a-1}(\rho)\quad&\text{if }1\in\{\rho(1),\ldots,\rho(a)\}\,,\end{cases}\]
and since $\cmp A$ holds these are the only two cases. Thus, since $D_{a,a-1}(\rho)=n$, we have $D_{a,a-1}(\rho^*)=1$ or $D_{a,a-1}(\rho^*)=2-n$; in either case $D_{a,a-1}(\rho^*)\equiv n$ \modn[n-1], but $D_{a,a-1}(\rho^*)\neq n$. Thus the mapping $\rho\mapsto\rho^*$ demonstrates that 
\begin{equation}\prob{D_{a,a-1}(\rho)= n\mid (D_{a,a-1}(\rho)\equiv n\mod{(n-1)})\wedge\cmp A}\leq\frac12\,.\label{half}\end{equation}
Thus, combining \eqref{comp} and \eqref{half}, we have
\begin{align*}\prob{(D_{a,a-1}(\rho)=n)\wedge\cmp A}&\leq\frac12\prob{(D_{a,a-1}(\rho)\equiv n\mod{(n-1)})\wedge\cmp A}\\
&\leq\frac{1}{2(n-1)}-\frac 12\prob{(D_{a,a-1}(\rho)=n)\wedge A}\end{align*}
and so, using \eqref{AnotA} and \eqref{Aholds},
\begin{align}p_{n-1;a,a-1}(n)&\leq\frac{1}{2(n-1)}+\frac 12\prob{(D_{a,a-1}(\rho)=n)\wedge A}\nonumber\\
&\leq\frac 3{4(n-1)}\,.\label{threequarters}\end{align}
Finally, combining \eqref{inout} and \eqref{threequarters},
\begin{align*}p_{n;a,a}(0)&<\frac{n-2a}{n}\cdot\frac 1{n-1}+\frac{2a}{n}\cdot\frac{3}{4(n-1)}\\
&=\frac{n-a/2}{n(n-1)}=\frac 1n+\frac{a-2}{2n(n-1)}\,.\end{align*}
This is sufficient to complete the proof for $a>2$. If $a=2$ and $n>5$, then in \eqref{Aholds} we have the better bound
\[\prob{(D_{a,a-1}(\rho)=n)\wedge A}\leq\frac1{n-3}\prob{A}=\frac{1}{n-3}\cdot\frac{a-1}{n-1}\leq\frac1{3(n-1)}\,.\]
This gives $p_{n-1;a,a-1}(n)\leq\frac 2{3(n-1)}$ instead of \eqref{threequarters}, and hence $p_{n;a,a}(0)\leq\frac{n-4/3}{n(n-1)}=\frac 1n-\frac{1}{3n(n-1)}$.
\end{proof}

\subsection{Degrees differing by more than $2$}
In this section we assume without loss of generality that $a>b$. First we give a simple bound which applies whenever the degrees are not equal.
\begin{lem}\label{diffeasy}Whenever $a>b$, we have $p_{n;a,b}(0)\leq\frac{1}{2(n-a-b+1)}$.\end{lem}
\begin{proof}First choose a random permutation $\tau$ and a random integer $r\in\{1\}\cup\{a+b+1,\ldots,n\}$, and define $\sigma$ by $\sigma(1)=\tau(r)$, $\sigma(r)=\tau(1)$, and $\sigma(i)=\tau(i)$ for all $i\neq 1,r$. For each $\tau$, at most one choice of $r$ will give $D_{a,b}(\sigma)=0$. Furthermore, $D_{a,b}(\sigma)>D_{a,b}(\tau)-\tau(1)$, so if $\sum_{i=2}^a\tau(i)\geq\sum_{i=a+1}^{a+b}\tau(i)$, no choice of $r$ will give $D_{a,b}(\sigma)=0$. Therefore 
\[p_{n;a,b}(0)\leq\frac 1{n-a-b+1}\probx{\sum_{i=2}^a\tau(i)<\sum_{i=a+1}^{a+b}\tau(i)}\,.\]
Since $a>b$, $\sum_{i=2}^a\tau(i)\geq\sum_{i=2}^{b+1}\tau(i)$, and by symmetry this is at least $\sum_{i=a+1}^{a+b}\tau(i)$ with probability at least $\frac 12$.\end{proof}

We next consider as a special case $b=1$.
\begin{lem}\label{b1}For any $a,n$ with $1<a<n-1$, we have $p_{n;a,1}(0)\leq\frac{n}{2(n-1)(a+1)}$.\end{lem}
\begin{proof}In order to have $D_{a,1}(\sigma)=0$ we must certainly have $\sum_{i=1}^{a+1}\sigma(i)$ even and $\sigma(a+1)>\sigma(i)$ for all $i\leq a$. The first has probability at most $\frac n{2(n-1)}$, by \lemma{parity}, and the second, independently, has probability $\frac 1{a+1}$.\end{proof}

The remaining cases we treat separately according to the parity of $n$. 
\begin{lem}\label{oddbigdiff}Whenever $n$ is odd and $a\geq b+2$, we have $p_{n;a,b}(0)\leq\frac{n+1}{4n(a+1)}$.\end{lem}
\begin{proof}
Choose a random permutation $\tau$ and integer $r\in[a+1]$, and set $\sigma(r)=\tau(a+1)$, $\sigma(a+1)=\tau(r)$ and $\sigma(i)=\tau(i)$ for $i\neq r,a+1$. 

We first consider $\prob{D_{a,b}(\sigma)\in\{0,(a-b)(n+1)\}}$. Since $n+1$ is even, and $D_{a,b}(\sigma)=D_{a,b}(\tau)+2\tau(a+1)-2\tau(r)$, this cannot happen unless $\sum_{i=1}^{a+b}\tau(i)$ is even. For a fixed $\tau$ with this property, at most one choice of $r$ gives $D_{a,b}(\sigma)\in\{0,(a-b)(n+1)\}$, since the difference in $D_{a,b}(\sigma)$ between two different values of $r$ is at most $2(n-1)$. Thus, using \lemma{parity},
\[\prob{D_{a,b}(\sigma)\in\{0,(a-b)(n+1)\}}\leq\frac 1{a+1}\probx{\sum_{i=1}^{a+b}\tau(i)\text{ even}}\leq\frac 1{a+1}\cdot\frac{n+1}{2n}\,.\]
Consider the mapping $\sigma\mapsto\sigma^*$, where $\sigma^*$ is given by $\sigma^*(i)\equiv n+1-\sigma(i)$. If $D_{a,b}(\sigma)=0$ then $D_{a,b}(\sigma^*)=(a-b)(n+1)$, and vice versa, so $\prob{D_{a,b}(\sigma)=0\mid D_{a,b}(\sigma)\in\{0,(a-b)(n+1)\}}=\frac12$. Thus
\[\prob{D_{a,b}(\sigma)=0}\leq\frac{n+1}{4n(a+1)}\,.\qedhere\]
\end{proof}

The argument for the even case is more complicated.
\begin{lem}\label{evenbigdiff}Whenever $n$ is even, $a\geq b+3$ and $b>1$, we have $p_{n;a,b}(0)\leq\frac{n}{4(n-1)a}$.\end{lem}
\begin{proof}
We choose a random permutation by first choosing a random value $s\in[n]$ and setting $\tau(s)=n$. We then extend $\tau$ to a random permutation and perform a random swap, chosen dependent on the value of $s$, which leaves the position of $n$ unchanged. We then argue that in each case $\prob{D_{a,b}(\sigma)=0\mid s}$ satisfies the required bound.
\begin{case}$s\not\in[a+b]$.\end{case}
In this case the problem reduces to showing that $p_{n-1;a,b}(0)\leq\frac{n}{4(n-1)a}$, and in fact a stronger bound holds by \lemma{oddbigdiff}.
\begin{case}$s\in[a]$.\end{case}
In this case we choose a random $r\in[a+1]\setminus\{s\}$, and set $\sigma(r)=\tau(a+1)$, $\sigma(a+1)=\tau(r)$, and $\sigma(i)=\tau(i)$ for $i\neq r,a+1$. For any $\tau$, we can have $D_{a,b}(\sigma)\in\{0,(a+1-b)n\}$ for at most one choice of $r$, and, since $n$ is even, this is only possible if $\sum_{i=1}^{a+b}\tau(i)$ is even, \ie if $\sum_{i=1}^{a+b}\tau(i)-n$ is even. Applying \lemma{parity} to $\{\tau(1),\ldots,\tau(a+b)\}\setminus\{\tau(s)\}$, which is a random $(a+b-1)$-subset of $[n-1]$, we have 
\[\prob{D_{a,b}(\sigma)\in\{0,(a+1-b)n\}}\leq\frac 1a\probx{\sum_{i=1}^{a+b}\tau(i)\text{ even}}\leq\frac 1a\cdot\frac{n}{2(n-1)}\,.\]
Now by considering the mapping $\sigma\mapsto\sigma^*$, where $\sigma^*$ is given by $\sigma^*(i)=n-\sigma(i)$ for $i\neq s$ and $\sigma^*(s)=\sigma(s)=n$, we have a bijection between those choices of $\sigma$ (subject to $\sigma(s)=n$) with $D_{a,b}(\sigma)=0$ and those with $D_{a,b}(\sigma)=(a+1-b)n$. Thus $\prob{D_{a,b}(\sigma)=0}=\frac12\prob{D_{a,b}(\sigma)\in\{0,(a+1-b)n\}}$, giving the required bound.
\begin{case}$s\in\{a+1,\ldots,a+b\}$.\end{case}
In this case we choose $q\neq s$ in $\{a+1,\ldots,a+b\}$, and a random $r\in[a]\cup\{q\}$. We then define $\sigma(r)=\tau(q)$, $\sigma(q)=\tau(r)$, and $\sigma(i)=\tau(i)$ for $i\neq q,r$. For any $\tau$ and $q$, we can have $D_{a,b}(\sigma)\in\{0,(a-1-b)n\}$ for at most one choice of $r$, and only if $\sum_{i=1}^{a+b}\tau(i)$ is even. We then proceed as in Case~2, and note that, using the same mapping $\sigma\mapsto\sigma^*$, we have $D_{a,b}(\sigma)=0$ if and only if $D_{a,b}(\sigma^*)=(a-1-b)n$. This gives a bound of $\frac{n}{4(n-1)(a+1)}$, which is stronger than required.\end{proof}

We now combine these results.
\begin{thm}\label{bigdiff}Whenever $a\geq b+3$, we have $p_{n;a,b}(0)\leq\frac1{n+1/2}=\frac 1n-\frac 1{n(2n+1)}$.\end{thm}
\begin{proof}If $b=1$ and $a\leq\frac{n-1}2$ then \lemma{diffeasy} gives $p_{n;a,b}(0)\leq\frac{1}{n+1}$. If $b=1$ and $a\geq\frac n2$ then \lemma{b1} gives $p_{n;a,b}(0)\leq\frac{n}{(n-1)(n+2)}\leq\frac 1{n+1/2}$ since $n>a+b\geq 5$.

If $a\leq\frac{n+7}4$ then \lemma{diffeasy} gives $p_{n;a,b}(0)\leq\frac{1}{n+1}$. If $b>1$ and $a\geq\frac{n+6}4$ then by \lemma{oddbigdiff} or \lemma{evenbigdiff} we have
\[p_{n;a,b}(0)\leq\frac{n}{4(n-1)a}\leq\frac{n}{(n-1)(n+6)}\leq\frac{1}{n+3}\,,\]
since again $n>5$.\end{proof}

\section{Main results}
In this section we conclude from the previous results that the local antimagic conjecture is true, and that the algorithm which labels the edges with a random permutation of $[m]$, checks whether the result is a local antimagic labelling, and repeats until one is found, will find one in expected time $O(m^2)$.
\begin{thm}\label{mainresult}For every graph $G$ with $m$ edges, none of which is isolated, labelling the edges of $G$ according to a random permutation of $[m]$ produces a local antimagic labelling with positive probability. Further, this probability is $\Omega(1/m)$.\end{thm}

\begin{proof}For every edge $xy$, we have $\prob{S_x=S_y}=p_{m;d(x)-1,d(y)-1}(0)$. Therefore, by \theorem{onetwo}, \theorem{equal} or \theorem{bigdiff}, we have $\prob{S_x=S_y}\leq 1/m$, and this inequality is strict unless $d(x)=d(y)=3$ and $m=5$. Since a graph with $5$ edges can have at most two vertices of degree $3$, the inequality is strict for at least $m-1$ edges. Thus the labelling is local antimagic with probability at least $1-\sum_{xy\in E(G)}\prob{S_x=S_y}>1-m/m$.

Further, we have $\prob{S_x=S_y}<1/m-\Omega(1/m^2)$ in all cases except where $d(x)=d(y)=\frac{m+1}2$ or $\abs{d(x)-d(y)}=1$. In the former case every vertex other than $x$ and $y$ has degree at most $2$, and so, provided $m>3$, only one edge has this property. In the second case, if $\min\{d(x),d(y)\}<m/4$ we still have $\prob{S_x=S_y}<1/m-\Omega(1/m^2)$, by \lemma{diffeasy}. Since $\sum_xd(x)=2m$, at most eight vertices have degree at least $m/4$, and so at most $28$ edges are between vertices of this type. Combining these facts, we have $\prob{S_x=S_y}<1/m-\Omega(1/m^2)$ for all but at most $29$ edges, and so a random labelling is local antimagic with probability $\Omega\bigl(\frac{m-29}{m^2}\bigr)=\Omega(1/m)$.\end{proof}

We therefore immediately have the following result, since the number of rounds of the algorithm described above is given by a geometric random variable with parameter $p=\Omega(1/m)$ and so the expected number of rounds is $O(m)$. In each round checking whether the labelling is local antimagic can be done in linear time, with one pass through the edges needed to compute all the sums $(S_x)_{x\in V(G)}$ and another pass to check that $S_x\neq S_y$ for each $xy\in E(G)$. Also, a random permutation of $[m]$ can be generated in linear time, as implemented in \cite{Dur64}.
\begin{cor}A local antimagic labelling of $G$ can be found using a random algorithm in expected time $O(m^2)$.\end{cor}

\section{Extensions and open problems}
In this section we discuss some natural extensions, and to what extent the methods of this paper apply.

\subsection{Changing the set of permitted labels}
Here we consider two extensions, suggested by the referee, which involve changing the permitted labels. 
The first is to require that the labels form a permutation not of $[m]$ but of another fixed set of $m$ consecutive positive integers. Many of the methods used to prove Theorem \ref{mainresult} translate directly to this setting, and we detail below the modifications which may be made to complete a proof of this generalisation. However, our methods do rely on the fact that the permitted labels are consecutive when transforming permutations and considering the effect mod $m$; they also rely on the fact that labels are positive so that the sum can only increase when more terms are added, \eg in \lemma{diffeasy}.

Suppose that $\sigma$ is a random permutation of $\{k,\ldots,n+k-1\}$, for some integer $k\geq 1$. The results of \sect{equal} still hold, since subtracting $k-1$ from each label does not change $D_{a,a}(\sigma)$. Inspecting the proof of \lemma{parity}, we see that the bounds also apply to random subsets of $\{k,\ldots,n+k-1\}$. Likewise, the proofs of Lemmas \ref{diffone}, \ref{difftwo}, \ref{diffeasy} and \ref{b1} translate directly to this setting. Lemmas \ref{oddbigdiff} and \ref{evenbigdiff} need slight modification, in that \lemma{oddbigdiff} covers the case $n+k-1$ odd (with $(a-b)(n+k)$ replacing $(a-b)(n+1)$ in the proof) and \lemma{evenbigdiff} the case $n+k-1$ even (with $(a+1-b)(n+k-1)$ replacing $(a+1-b)n$). However, all cases are still covered, and consequently \theorem{bigdiff} still holds. Thus the only obstacle to proving that a neighbour-distinguishing labelling using these labels exists is that the proof of \theorem{onetwo} does not translate to this setting. We therefore need the following result.

\begin{thm}\label{newonetwo}Let $k$ and $n$ be fixed positive integers, and $a$ and $b$ be positive integers such that $b\in\{a-2,a-1\}$ and $a+b<n$. If $\sigma$ is a random permutation of $\{k,\ldots, n+k-1\}$ then $\prob{D_{a,b}(\sigma)=0}<\frac 1n$.\end{thm}
\setcounter{case}{0}
\begin{proof}
We distinguish three cases.
\begin{case}$b=a-1$.\end{case}
In this case the arguments of \lemma{diffone} still give $\prob{D_{a,a-1}(\sigma)\equiv 0\mod n}=\frac 1n$, and so it suffices to show that for some permutation $\sigma$ we have $D_{a,a-1}(\sigma)\equiv 0$ \modn, but $D_{a,a-1}(\sigma)\neq 0$. If $a$ is odd, choose $\sigma(1)\equiv 0$ \modn; if $a$ is even choose $\sigma(1)\equiv -1$ \modn. The remaining integers $\{k,\ldots,n+k-1\}\setminus\sigma(1)$ contain at least $\frac{n-3}2$ disjoint pairs of consecutive integers, and $a-1\leq\frac{n-3}2$. Take $a-1$ such pairs and for each $i\in[a-1]$ assign the integers from a pair to $\sigma(i+1)$ and $\sigma(i+a)$, with the higher integer going to the former if $i$ is odd and the latter otherwise. Then $D_{a,a-1}(\sigma)=\sigma(1)$ if $a$ is even and $D_{a,a-1}(\sigma)=\sigma(1)+1$ if $a$ is odd; in either case we have the required property.
\begin{case}$b=a-2$ and $n$ is odd.\end{case}
In this case the arguments of \lemma{difftwo} give $\prob{D_{a,a-2}(\sigma)\equiv 0\mod n}=\frac 1n$, and so we may proceed as in Case 1. Set $d=0$ if $a$ is even and $d=-1$ if $a$ is odd. Choose $\sigma(1)$ and $\sigma(2)$ such that $\sigma(1)+\sigma(2)\equiv d$ \modn and $\sigma(1)\in\{k,n+k-1\}$. This is always possible, since we may choose $\sigma(1)=k$ unless $2k\equiv d$ \modn, in which case $2(n+k-1)\not\equiv d$. The remaining integers in $\{k,\ldots,n+k-1\}$ contain at least $\frac{n-4}2\geq a-2$ consecutive pairs; we assign these to $\sigma(i+2)$ and $\sigma(i+a)$ as in Case 1.
\begin{case}$b=a-2$ and $n$ is even.\end{case}
Set $H_1=\{k,\ldots,n/2+k-1\}$ and $H_2=\{n/2+k,\ldots,n+k-1\}$. Write $p_{\text{even}}$ for the probability that a random subset of $[n]$ of size $2a-2$ has even sum; note that 
\[\prob{D_{a,a-2}(\sigma)\text{ even}}=\prob{\abs{\{\sigma(1),\ldots,\sigma(2a-2)\}\cap H_1}\text{ even}}=p_{\text{even}}\,.\]
Set $S=\{\sigma:D_{a,a-2}(\sigma)\equiv 0\mod n\}$. By the arguments of \lemma{difftwo} we have 
\begin{equation}\prob{\sigma\in S}=\frac 2np_{\text{even}}\,;\label{inS}\end{equation}
we may assume $p_{\text{even}}\geq\frac12$ since otherwise \eqref{inS} gives the required inequality.

Write $A_{\sigma}=\{\sigma(1),\ldots,\sigma(a)\}$, $B_\sigma=\{\sigma(a+1),\ldots,\sigma(2a-2)\}$ and define $T=\{\sigma:\abs{A_{\sigma}\cap H_1}=\abs{B_{\sigma}\cap H_1}+1\}$.
Note that $\prob{\sigma\in T}<\prob{\abs{\{\sigma(1),\ldots,\sigma(2a-2)\}\cap H_1}\text{ odd}}=1-p_{\text{even}}$. Choose a random $\tau\in T$ and $r\in [n/2]$, and define the permutation $\tau^*$ by setting $\tau^*(i)=\tau(i)$ if $\tau(i)\in H_2$ and $\tau^*(i)$ to be the element of $H_1$ congruent to $\tau(i)+r$ \modn[n/2] if $\tau(i)\in H_1$. Note that $\tau^*$ is a uniformly random element of $T$, and $D_{a,a-2}(\tau^*)\equiv D_{a,a-2}(\tau)+r$ \modn[n/2], so for each $\tau$ at most one choice of $r$ gives $\tau^*\in S$. Thus 
\begin{equation}\prob{\sigma\in S\cap T}\leq\frac 2n\prob{\sigma\in T}<\frac 2n(1-p_{\text{even}})\label{SandT}\,.\end{equation}
Let $\rho$ be a permutation in $S\setminus T$, and consider the mapping $\rho\mapsto\rho^*$ given by 
\[\rho^*(i)=\begin{cases}\rho(i)+n/2&\quad\text{if }\rho(i)\in H_1\\
\rho(i)-n/2&\quad\text{if }\rho(i)\in H_2\,.\end{cases}\]
The map $\rho\mapsto\rho^*$ is a bijection on $S\setminus T$. Since $\rho\not\in T$, we must have 
\[\abs{A_{\rho}\cap H_1}-\abs{B_{\rho}\cap H_1}\neq\abs{A_{\rho}\cap H_2}-\abs{B_{\rho}\cap H_2}\,,\]
and hence $D_{a,a-2}(\rho)\neq D_{a,a-2}(\rho^*)$. Consequently we have, using \eqref{inS} and \eqref{SandT},
\begin{align*}\prob{D_{a,a-2}(\sigma)=0}&\leq\prob{\sigma\in S}-\frac12\prob{\sigma\in S\setminus T}\\
&=\frac12\bigl(\prob{\sigma\in S}+\prob{\sigma\in S\cap T}\bigr)\\
&<\frac12\Bigl(\frac2np_{\text{even}}+\frac2n(1-p_{\text{even}})\Bigr)=\frac1n\,,
\end{align*}
as required.\end{proof}

Combining \theorem{newonetwo} with the observations made above, if the edges of an $m$-edge graph with no isolated edges are labelled with a random permutation of $\{k,\ldots,n+k-1\}$ then every pair of adjacent edges is distinguished with probability at least $1-\frac1m$, and with strictly greater probability except in the case of two adjacent vertices of degree $3$ when $m=5$. Thus we have the following strengthening of the local antimagic conjecture.
\begin{thm}For every graph $G$ with $m$ edges, none of which is isolated, and for any positive integer $k$, the edges of $G$ may be labelled with a permutation of $\{k,\ldots,m+k-1\}$ in such a way that the vertex sums distinguish all pairs of adjacent vertices.\end{thm} 

A second and much more general extension, by analogy with the well-studied concept of list colouring of graphs introduced by Erd\H{o}s, Rubin and Taylor (\cite{ERT79}), is to give each edge a list of $m$ permitted positive integer weights. We say that a labelling of edges is \textit{feasible} if each edge has a label from its own list, and no label is used more than once. Is there always a feasible labelling for which the vertex sums distinguish all neighbouring pairs?

Since the individual lists need not have any particular structure, the details of our proof certainly do not transfer to this setting. However, there is a more fundamental obstacle to a probabilistic approach, as the lists may intersect in arbitrary ways, meaning that individual labels in a random feasible labelling will typically not be uniformly distributed, and their distributions will vary depending on the lists of other edges. 

However, we conjecture that if these difficulties could be overcome, the basic proof strategy of this paper would be viable in the following precise sense.
\begin{con}\label{lists}Let $G$ be a graph with $m$ edges, none of which is isolated, and assign to each edge a list of $m$ positive integers. Let $f:E(G)\to\mathbb Z^+$ be a uniformly random feasible assignment. Then for every edge $vw\in E(G)$, $\probb{\sum_{e\ni v}f(e)=\sum_{e\ni w}f(e)}<\frac 1m$.\end{con}
\conj{lists} holds if the lists are disjoint, since then a random feasible assignment simply involves selecting a random permitted label independently for each edge. Randomly labelling each edge in turn, when we reach the last edge meeting exactly one of $v$ and $w$, there is a unique integer $r$ such that assigning $r$ to that edge will give $\sum_{e\ni v}f(e)=\sum_{e\ni w}f(e)$. Further, there is a positive probability that $r$ is not a permitted label for that edge, either because there are more than $m$ possible values of $r$, or because there are at most two edges meeting exactly one of $v$ and $w$ (in which case no labelling with distinct positive integers can fail to distinguish them).

\subsection{Distinguishing more pairs of vertices}
We conclude with the following conjecture, which seems the natural next step from this result towards the antimagic labelling conjecture. 
It is in the same spirit as the concept of distant irregularity strength, introduced in \cite{Prz13} as intermediate between irregularity strength and the 1-2-3 conjecture, where the requirement is to distinguish every pair of vertices with distance at most $r$ but labels may be repeated ad lib.
\begin{con}\label{twostep}For any graph $G$ with $m$ edges, none of which is isolated, there is a labelling using every label from $[m]$ exactly once for which the sums $(S_v)_{v\in V(G)}$ distinguish every pair of vertices $v,w$ with graph distance $d(v,w)\leq 2$.\end{con}
An approach that relies solely on the union bound is not sufficient to prove Conjecture \ref{twostep}. To see this, consider the graph $K_{2,n}$ which has $n$ vertices of degree $2$, no pair of which share an edge but every pair of which have distance $2$, and $2n$ edges. It is easy to see that for a random permutation the probability of two given vertices of degree $2$ having the same sum is $\Theta(1/n)$ (whenever $1\leq a,b,c\leq n/2$, if the edges meeting one vertex have labels $a$ and $a+2b+c$ while those meeting the other vertex have labels $a+b$ and $a+b+c$ then the sums are equal); since there are $\binom n2$ pairs of degree-$2$ vertices, these probabilities sum to $\Theta(n)$.

\acknowledgements
The author was supported by the European Research Council (ERC) under the European Union's Horizon 2020 research and innovation programme (grant agreement No.\ 639046), and thanks the anonymous referee for some helpful comments.

\nocite{*}
\bibliographystyle{abbrvnat}
\bibliography{antimagic}
\label{sec:biblio}

\end{document}